\documentclass[review,times]{elsarticle}  
\usepackage{amsmath,amssymb,amsfonts,amsthm}

\newtheorem{theorem}{Theorem}

\newcommand{\sepa}{\sep}
\usepackage{xcolor}\newcommand{\rmrk}[1]{{\color{red}{[#1]}}}
\renewcommand{\rmrk}[1]{}
\newcommand{\marklast}{}
\newcommand{\be}[1]{\begin{equation}\label{#1}}
\newcommand{\ee}{\end{equation}}

\newcommand{\myfrac}[2]{ {#1}/{#2} }
\newcommand{\myfraca}[2]{ {(#1)}/{#2} }

\newcommand{\errmat}{{\varepsilon}}
\newcommand{\errmatall}{\errmat}

\newcommand{\errtabb}[1]{$\errmatall={#1}$}
\newcommand{\allerr}[1]{\errtabb{#1}}

\newcommand{\pspan}[1]{{ {\mathrm{span}} \left( {#1} \right) }}

\newcommand{\nnode}{{n_{p}}}

\newcommand{\bfn}{{\mathbf{n}}}
\newcommand{\bfu}{{\mathbf{u}}}
\newcommand{\bfx}{{\mathbf{x}}}

\newcommand{\trp}{{\scriptscriptstyle\mathsf{T}}}

\newcommand{\ndim}{d}

\newcommand{\mquad}[1]{$m_{#1}$}
\newcommand{\mquadopt}[1]{$m_{#1}^{\mathrm{opt}}$}
\newcommand{\stquad}[1]{$s_{#1}$}
\newcommand{\stquadopt}[1]{$s_{#1}^{\mathrm{opt}}$}

\newcommand{\KK}{\mathbf{K}}
\newcommand{\bfA}{\mathbf{A}}
\newcommand{\bfB}{{\mathbf{B}}}

\newcommand{\AL}{\bfA^{\textrm{L}}}
\newcommand{\dinner}{d_{\mathrm{inner}}}
\newcommand{\Smplx}{{\mathcal{S}}}
\newcommand{\bfBnum}{{\overline{\mathbf{B}}}}
\newcommand{\bub}{\eta} 
\newcommand{\al}{a}
\newcommand{\Permfun}[1]{S\!\left[{#1}\right]}

\newcommand{\CFL}{\sigma}
\newcommand{\CFLnum}{{\bar\sigma}}

\newcommand{\toprule}{\hline}
\newcommand{\midrule}{\hline}

\begin{document}

\newcommand{\mytitle}{Quadrature rules for the mass and stiffness matrices of finite elements for the wave equation on the 2-, 3- and 4-simplex}
\newcommand{\myrunning}{Quadrature rules for mass and stiffness matrices}
\newcommand{\mysumm}{
Mass lumping enables explicit time stepping for finite-element discretisation of the wave equation if the resulting quadrature weights are positive and accuracy is preserved. New rules for elements of degree two and three on the 4-simplex  are presented. Numerical quadrature for the stiffness matrix can be more efficient than exact evaluation if it requires fewer nodes. For the latter, new rules in two and four space dimensions for the lower-degree elements on the simplex were found, as well as some additional results for three dimensions.
}

\newcommand{\mykeywords}{
  quadrature \sepa simplex \sepa node patterns \sepa polynomial \sepa finite elements \sepa wave equation}
\newcommand{\mymsc}{65D32 \sepa 65N30} 

\newcommand{\myname}{W.~A.~Mulder}
\newcommand{\myemaildelft}{w.a.mulder@tudelft.nl}
\newcommand{\delftaddress}{Department of Geoscience \& Engineering,
Faculty of Civil Engineering and Geosciences,
Delft University of Technology,
Delft, The Netherlands.}
\newcommand{\myorcid}{Orcid: 0000-0001-7020-9297}

\begin{frontmatter}

  \title{\mytitle}

 \author{\myname}
 \affiliation{organization={Delft University of Technology, Department of Geoscience \& Engineering, Faculty of Civil Engineering and Geosciences}, addressline={Stevinweg 1}, city={Delft}, postcode={2600~GA},country={The Netherlands}}

\begin{abstract} \mysumm \end{abstract}


\end{frontmatter}

\section{Introduction}\label{sec:intro}

When solving the wave equation, finite elements can be more efficient than finite differences if the mesh scales with the local (shear) velocity and if the element faces follow the discontinuities of the material properties \cite[e.g.]{ref:mulder1996,ref:kononov,ref:zhebel2014}.
Mass lumping avoids the inversion the large sparse mass matrix and allows for explicit time stepping, resulting in a scheme with a computational structure similar to the finite-difference method. Straightforward lumping for standard polynomial  simplicial elements leads to a loss of spatial accuracy, which can be restored by the use of higher-degree polynomials in the interior of the element, as demonstrated for triangles \cite{ref:tordjmana,ref:tordjmanc} and tetrahedra \citep{ref:mulder1996}.

The construction of mass-lumped elements involves a choice of quadrature nodes and the requirement that quadrature be exact for a certain set of polynomials.
The solution of the resulting polynomial system of equations provides the actual node positions and quadrature weights, but its complexity explodes rapidly when the number of nodes and polynomial degree increases.

So far, triangular elements could be constructed for degrees
2 \citep{ref:crouzeix1973,ref:zien1973}, 
3 \citep{ref:tordjmanb,ref:tordjmanc},
4 \citep{ref:mulder1996},
5 \citep{ref:chin},
6 \citep{ref:mulder2013},
7 and 8 \citep{ref:liu,ref:cui2017}. The last two papers also contain degree-9 elements,
but the one in \citep{ref:cui2017} has degree 10 instead of 9 on the edges
whereas the element in \citep{ref:liu}, when used as an initial guess,
does not seem to converge to a solution of the quadrature equations when very high extended precision is used in a Newton-type root-finding method.
Nevertheless, the proposed element appears to be sufficient for applications with double-precision computations.

Tetrahedral elements were initially found for degree 2 \citep{ref:mulder1996} and 3 \citep{ref:chin}.
A sharper accuracy criterion was proposed in \cite{ref:geevers_new}, leading to new tetrahedral elements of degree 2 and 3, as well as several of degree 4.
In that paper, the degree-4 elements had a subset of the polynomials with degree higher than 4.
Here, elements of that type will not be considered.
The same sharper accuracy criterion also led to simpler triangular elements of degree 5 and 6 \citep{ref:mulder2022,ref:mulder2024perf}.

Explicit time stepping for the wave equation involves repeated multiplications of the stiffness matrices and the solution vectors. Evaluation on the fly instead of full assembly of these matrices can speed up the computations \cite[e.g.]{ref:sham2016perf}.
Numerical instead of exact quadrature for the stiffness matrices can further reduce the required number of operations \cite{ref:geevers_stiff}.

\begin{sloppypar}
The accuracy of the time stepping scheme is another topic.
One option is higher-order time stepping with the Lax–Wendroff approach \cite{ref:laxwendroff}, also known as the Cauchy–Kowalevsky\cite{ref:kowalevsky}
or Dablain \cite{ref:dablain} or modified-equation method \cite{ref:shubin}.
Another with low cost is dispersion correction of the second-order scheme \cite{ref:stork,ref:koene,ref:mulderdisp} \rmrk{more}.
\end{sloppypar}

The goal of the current paper is to extend mass lumping to 4D.
In addition, dedicated numerical quadrature rules for the stiffness matrices in two to four dimensions are constructed, using the same criteria as proposed in \cite{ref:geevers_stiff}, where they were applied to tetrahedra. New rules for  the stiffness matrices are provided for the triangle and 4-simplex, as well as some additional ones for the tetrahedron.
These rules depend on the basis functions that follow from the construction of the mass matrix. The number of required nodes tends to be less than with the generic quadrature rules
that are exact for polynomials up to a given degree, a recent collection of which can be found in \cite{ref:worku2026}.

Section~\ref{sec:method} reviews the construction of quadrature rules for polynomial basis functions and
a symmetric set of nodes, both for the mass and for the stiffness matrices.
Section~\ref{sec:results} contains the results for lower-degree elements on the 2-, 3- and 4-simplex.
A large number of potential candidates were considered. For the lower-degree polynomial systems of quadrature equations, the existence or non-existence of positive solutions could usually be demonstrated, but for higher degrees and higher dimensions, an exhaustive treatment proved impossible due to the computational complexity of many cases.
Section~\ref{sec:conc} concludes.

\rmrk{check literature again! Gauss ? Walkington? many more \citep{ref:hammer1956}, \citep{ref:zhang2009}}

\section{Quadrature reviewed}\label{sec:method}

Mass lumping replaces the mass matrix by a diagonal matrix. The values on the diagonal are equal, or proportional to, numerical quadrature weights. For finite-element discretisations of the acoustic or elastic wave equation, these weights should be positive for stable explicit time stepping while preserving the spatial accuracy.

To find suitable weights, which are proportional to numerical quadrature weights, polynomial basis functions are chosen. The symmetric node pattern follows from a necessary condition for unisolvence \cite{ref:mulderuni}. Requiring the rule to be exact for certain polynomials leads to a polynomial system. Its solution, if it exists, provides the weights and the parameters for the node positions. Solving this system is the biggest hurdle in the construction of the elements.

In what follows, the choice of polynomials, symmetry classes on the simplex in two to four dimensions,
and the accuracy criteria for quadrature of the mass and stiffness matrices are reviewed.

\subsection{Polynomials}
Straightforward lumping of the mass matrix with a regular distribution of nodes on the simplex causes a loss of spatial accuracy. This can be amended by using polynomials of a higher degree in the interior of the $m$-faces.

A hierarchical polynomial space can be defined with the aid of bubble functions \citep[e.g.]{ref:mulderuni}.
The barycentric coordinates on the simplex are defined as $(x_0,x_1,\ldots,x_\ndim)$
with $x_0=1-\sum_{k=1}^{\ndim} x_k $. A bubble function on an $m$-face, starting at the edges for $m=1$, is
given by $\bub_{m}=\prod_{k=0}^m x_k$ on one face.
The bubble functions on the other faces follow from $\Permfun{\bub_m}$, where the permutation operator $\Permfun{f}$ acting on a function $f(x_0,x_1,\ldots,x_\ndim)$ produces a set of functions from all permutations of its arguments.
The same notation will be used for a set of functions, with the operator applied to each element.

The monomials of degree $p\ge 0$ on one $m$-face can be defined as
\be{eq:monomials}
P_p(x_0,\ldots,x_m) = \left\{
    \prod_{k=0}^m x_k^{\ell_k}  \ \bigg \vert\
      \sum_{k=0}^{m} \ell_k = p,\ \ell_k \in {\mathbb{N}}_0
\right\}.
\ee
Note that $(x_1,\ldots,x_m)$ defines a point on that $m$-face, but $x_0$ is defined on the $\ndim$-simplex.
For $p<0$, the set is taken as empty.
With that, the polynomial space with different degrees on the $m$-faces is defined as
\be{eq:polspace}
U(p_0,p_1,p_2,\ldots,p_\ndim)=\bigcup_{m=0}^\ndim \Permfun{V_m},
\ee
where, for $m\ge 0$,
\be{eq:vv}
V_m=\left\{ \bub_m \right\}\otimes P_{p_m-m-1}(x_0,x_1,\ldots,x_m).
\ee
The space $U$ is face-conforming, leading to finite-element solutions that are continuous across elements.
In this paper, elements of degree $p_1$ on the edges and including the vertices are considered, with $p_0=p_1$ as the base degree of the element, so we must have $P_{p_1}\subseteq U(p_1,p_2,\ldots,p_\ndim)$.
In addition, subsets of the higher-degree subspaces for $m>1$ of $U$, as used in \citep{ref:geevers_new}, are not considered.

Note that $\Permfun{V_\ndim}=V_\ndim$ in the interior of the simplex
and that $V_0=\{x_0^{p_0}\}$. The latter can be replaced by $V_0=\{x_0\}$ in the definition of $U$ in
equation~\eqref{eq:polspace} if $p_0=p_1$ and $P_{p_1}\subseteq U$, as proven in \ref{app:lower}.
Lower degrees lead to simpler expressions in the systems of polynomial equations that have to be solved
to find the quadrature rules.

\subsection{Quadrature for the mass matrix}
Numerical quadrature requires a set of nodes on the simplex. Here, only symmetric arrangements are considered.
A symmetric set of nodes can be divided into equivalence classes by selecting subsets
for which two or more of the barycentric coordinates $x_k$ are equal \citep[e.g.]{ref:keast1986,ref:mulderuni}.
An interior node for the $\ndim$-simplex
in class $\bfn=[n_1,n_2,\ldots,n_k]$, with $n_1\ge n_2\ge\ldots\ge n_k>0$ and $\sum_{\ell=1}^{k} n_\ell=\ndim+1$,
has a position $\bfx=(a_1,a_1,\ldots,a_1,a_2,\ldots,a_2,\ldots,a_k,\ldots,a_k)$ in barycentric coordinates
with node parameter $a_\ell$ repeated $n_\ell$ times and with distinct $a_\ell>0$ for $\ell=1,\ldots,k$.
Given such a node $\bfx$, the other nodes in the same symmetry class follow from the permutations of its coordinates.

The nodes on the $m$-faces can be partitioned into  equivalence classes in a similar manner, for $m$ space dimensions.
They can then be embedded in $\ndim$ dimensions by zero padding, setting the missing $\ndim-m$ coordinates to zero
and permuting the coordinates to obtain all nodes for that class.

Table~\ref{tab:classd} lists the symmetry classes on the simplex for dimensions 0 to 4.
A node pattern $\KK$ characterises
the number of times $K_j$ that a generating node occurs in a given class enumerated by an index $j$,
which runs over the equivalence classes from lower to higher dimensions.
The classes $\bfn$ are represented by the sets of integer partitions of $m+1$, ordered
in the usual way from high to low and short to long, and the index $j$ follows that order.
If $\bfn$ has length 1 for some index $j$, $K_j=1$, that is, there can only be one generating node in that dimension.
For $m<\ndim$, the zero padding and permutations will then still produce a total of $\binom{\ndim+1}{m+1}$ nodes.

In this paper, $p_0=p_1\ge 1$. As a result, the vertices are always included and $K_1=1$.
On edges, the midpoint is included if $p_1$ is even, that is $K_2=1$ if $p_1$ even and $K_2=0$ otherwise.
There are $2 K_3$  other points on each edge, where $K_3=(p_1-K_2-K_1)/2$, and there are $\binom{\ndim+1}{2}$ edges.

Given the set of monomials in $U$, a necessary condition for unisolvence automatically defines the node pattern,
as conjectured in \citep{ref:cui2017}, proven for triangles in \citep{ref:mulder2022},
for the $\ndim$-simplex in \citep{ref:mulderuni} and further generalised in \citep{ref:xie2025}.
The symmetry classes of the monomials should match those of the generating nodes.

With the polynomial space $U$ and associated node pattern,
we can attempt to find a quadrature rule by requiring it to be exact
for all polynomials in  $P_{p_1-2}\otimes U$ \citep{ref:geevers_new}.
The resulting system of polynomial equations contains the quadrature weights and
node parameters as unknowns. It can be inconsistent, zero-dimensional, or positive-dimensional,
meaning there are no solutions, a finite number of complex solutions, or an infinite number of them.
Here, only real and positive solutions are of interest, with the obvious additional constraints that nodes
should stay in there symmetry class and multiple generating nodes in the same class should remain distinct,
including there permutations.
Given the computational complexity of these problems, finding quadrature rules for higher degrees
is difficult.

\rmrk{the following is not used ...? what about $a=(1-\sqrt{1-\alpha})/2$ for class 3? occurs as
$a$ and $1-a$, odd powers cancel. Also $b=(1-\sqrt{1-8 \beta})/4$  for class 9}

The node parameters are denoted by $\al_{jk\ell}$ for a node at $K_j$, with $k=1,\ldots,K_j$ and $\ell=1,\ldots,\ell_{\max}$ and $\ell_{\max}\le \ndim$, where the upper bound $\ndim$ may not be reached if
some the nodes parameters in the equivalence class are the same.
The corresponding weights could be enumerated by $w_{jk}$, but here, they are just numbered consecutively, starting with $w_1$ for the vertices.

\begin{table}[ht]
  \caption{Classes for node patterns $\KK$ up to 4D.
    The list contains the index $j$ for $K_j$ followed by the symmetry class of the node.}\label{tab:classd} 
{ \footnotesize
  \setlength{\tabcolsep}{4pt} 
   \begin{center}\begin{tabular}{c l l l l l l l}
$m$-face & \multicolumn{5}{l}{index:class} \\
\hline
0&\phantom{1}1:[1]\\
1& \phantom{1}2:[2] &\phantom{1}3:[1,1]\\
2&\phantom{1}4:[3] &\phantom{1}5:[2,1] &\phantom{1}6:[1,1,1]\\
3&\phantom{1}7:[4] &\phantom{1}8:[3,1] &\phantom{1}9:[2,2] & 10:[2,1,1] & 11:[1,1,1,1] \\
4& 12:[5]& 13:[4,1] & 14:[3,2] & 15:[3,1,1]& 16:[2,2,1]& 17:[2,1,1,1]& 18:[1,1,1,1,1]\\
\end{tabular}\end{center} }
\end{table}

The mass matrix $\bfA$ has entries $a_{i,j}=\int_{\mathcal{S}} \psi_i \psi_j \,d\bfx$.
The stiffness matrices $\bfB^{k,\ell}$ have entries
$b_{i,j}^{k,\ell}=\int_\Smplx \frac{\partial\psi_i }{\partial x_k}\frac{\partial\psi_j}{\partial x_\ell} \,d\bfx$
for $k,\ell=1,2,\ldots,\ndim$,
where the integration is carried out over the reference simplex $\Smplx$ with $0\le x_j \le 1$ for $j=0,1,\ldots,\ndim$.
The lumped mass matrix $\AL$ is diagonal with values equal to the row sums of $\bfA$, which are the same as the quadrature weights.
\rmrk{reference for proof?}

The performance of a finite-element discretization of the acoustic or elastic wave equation
depends on the degree of the polynomial basis, the number of nodes, and the maximum allowable time step,
which is limited by the Courant-Friedrichs-Lewy (CFL) number \citep{ref:cfl,ref:mulder2013}.
A second-order time-stepping schemes involves the operator $\left(\AL\right)^{-1} \sum_k \bfB^{k,k}$. Its largest
eigenvalue $\lambda_{\max}$ determines the CFL number for the reference simplex: $\CFL=(2/\dinner)/ \sqrt{\lambda_{\max}}$.
Here, $\dinner=2/(\ndim+\sqrt{\ndim})$ is the diameter of the inscribed sphere for the reference simplex.
The results provides a reasonable estimate, although in \citep{ref:geevers_new,ref:mulder2024perf},
the power method provided slightly larger results.

If the polynomial system of equations has multiple solutions,
it may be expected that a larger CFL number will lead to better performance
and one could select the element with the largest estimated CFL number, if the approximation error
does not change too much among various options. In the case of infinitely many solutions, which typically occurs
if there are less equations than unknowns in the polynomial system, numerical maximisation for the CFL number
may be useful. 
Another performance measure is the Lebesgue constant 
$\Lambda={\max}_{ {\bfx} \in { \Smplx} } \sum_{j=1}^{\nnode} \vert  \psi_j(\bfx) \vert $,
where $\bfx$ are the barycentric coordinates in the reference simplex $\Smplx$ and the $\psi_j(\bfx)$ the Lagrange interpolating functions serving as basis functions, defined on $\nnode$ nodes. Smaller measures indicate better accuracy, when comparing elements of the same degree.

\subsection{Quadrature for the stiffness matrix}
Given a quadrature rule for the mass matrix, the Lagrange basis functions $\psi_i$ can be constructed.
Exact integration then provides the stiffness matrices on the reference element.
Since the matrix-vector multiplication with these matrices are the most costly part of a finite-element
code with explicit time stepping \citep{ref:sham2016perf},
numerical integration can be less costly.
In \cite{ref:geevers_stiff}, we proposed the use of numerical quadrature with
the requirement that it be exact for polynomials
in the set $P_{p_1-1}\otimes D U $ and is spurious-free.
Here, $D U$ the set of polynomials obtained by taking partial derivatives of the
polynomials in $U$ with respect to the coordinates $x_1,x_2,\ldots,x_\ndim$.
A spurious-free quadrature rule has the property that $\nabla u=\mathbf{0}$ with $u\in U$, evaluated on all the quadrature nodes, only has the constant solution.
This can be verified for $u=\sum_{j=1}^\nnode u_j \psi(\bfx)$ by considering the null space  in terms of $u_j$ of the set of equations $\myfrac{\partial u}{\partial x_k}(\bfx_j)=0$ with $k=1,\ldots,\ndim$ and $j=1,\ldots,\nnode$.
Likewise, the stress tensor $\nabla \bfu+(\nabla \bfu)^\trp$  in the elastic wave equation should
vanish only for rigid motions when evaluated at all the nodes.
Among the results, there were several cases were the spurious-free condition was not met, usually for the quadrature rules with a relatively small number of nodes.

To find potential candidates,  all symmetric node patterns up to a maximum number of nodes are considered,
in contrast to the mass matrix where the node pattern follows from the chosen set of polynomials.
Each pattern leads to a system of polynomial equations that may or may not have one or more solutions.
If there is a solution, it should be verified that it has positive  weights and is spurious-free.
The pattern with the smallest number of nodes is expected to have the smallest computational cost.

If there are multiple or infinitely many solutions for one pattern or several patterns with the same number of nodes,
additional requirements may be imposed.
For the  mass matrix, one could for instance maximise the CFL number,
as in \citep{ref:mulder2013} and \citep{ref:mulder2022}.
For the stiffness matrix, the difference between the directly evaluated stiffness matrices $\bfB^{k,\ell}$ and
the result ${\bfBnum}^{k,\ell}$ of numerical quadrature can be minimised.
One option is to minimise 
$\| {\bfB}^{1,1} - {\bfBnum}^{1,1} \|^2$, in some norm, which is the same
as minimising $\sum_{k=1}^\ndim  \| {\bfB}^{k,k} - {\bfBnum}^{k,k} \|^2$, because the matrices for one value of $k$ can be found from another by
a permutation. This is suitable for the scalar wave equation.
For the elastic case,
one could consider the minimisation of
$\errmatall^2=\sum_{k=1}^\ndim\sum_{\ell=1}^\ndim  \| {\bfB}^{k,\ell} - {\bfBnum}^{k,\ell} \|^2$.
Because of symmetry, this simplifies to
\be{eq:bsymerr}\errmatall^2 =
 \sum_{k=1}^\ndim\sum_{\ell=1}^\ndim  \| \bfB^{k,\ell}-{\bfBnum}^{k,\ell} \|^2 =
 \ndim \| \bfB^{1,1}-{\bfBnum}^{1,1} \|^2 +
 \ndim(\ndim-1)\| \bfB^{1,2}-{\bfBnum}^{1,2} \|^2.
\ee
Here, the Frobenius norm $\| \bfB \|=\left[ {\mathrm{trace}}\left( \bfB^\trp \bfB \right)\right]^{1/2}$ is
chosen. Recall that the squared Frobenius norm can also be expressed
as $\| \bfB \|^2=\sum_{i=1}^{\nnode}  \sum_{j=1}^{\nnode} \vert b_{i,j} \vert^2$ and equals the sum of squared eigenvalues.

\section{Results}\label{sec:results}

Mathematica\textsuperscript{\textregistered} \cite{ref:Mathematica} was used to generate and solve the polynomial systems and to carry out the
optimisation. Extended precision was used throughout, but the results are only shown with 24 digits.
Several workarounds for Mathematica's idiosyncrasies were implemented.
Many of the optimisation results were obtained in different ways, but it cannot be  guaranteed
that the global minimum or maximum was found.
Also, in the many cases considered, the existence of non-existence of solutions of the polynomial system could not be established.

\rmrk{ (Hammer and Stroud 1956 ?) Degree 2 has $w_1=1/(\ndim+1)!$ on ??? Dunavant? check for known, \cite{ref:pap}, Zhang??  \cite{ref:worku2026} }

\rmrk{PROBLEM: cannot trust min and max optimisation results ... FindMinimum is local, NMinimize often as well:
  ``Strictly speaking, Nelder–Mead is not a true global optimization algorithm; however, in practice it tends to work reasonably well for problems that do not have many local minima.''
  }

\rmrk{{\textbf{not}} exhaustive, some cases are, mention which, state open problems?}

\subsection{Linear element}

The linear degree-1 element in $\ndim$ dimensions has quadrature weights $1/(\ndim+1)!$ on each of the $\ndim+1$ vertices.
For barycentric coordinates $x_k$, $k=0,1,\ldots,\ndim$,
with $x_0=1-\sum_{k=1}^{\ndim} x_k$,  
the basis functions are $\psi_k(x_1,\ldots,x_\ndim)=x_k$.
The CFL number on the reference simplex is
$\CFL=(\ndim+\sqrt{\ndim})/ (\ndim+1)$,
equalling $1$, $1.138$, $1.183$, $1.2$, $1.206$, $1.20707$ and $1.20572$
for $\ndim=1,2,...,7$ and decreasing for large $\ndim$. The maximum is reached for $\ndim=6$.
For the numerical quadrature of the stiffness matrix, the centroid at $x_k=1/(\ndim + 1)$ for $k=0,1,\ldots,\ndim$
with weight $1/\ndim!$ suffices.

\subsection{Triangle}

Table~\ref{tab:twoa} lists a known triangular element \citep{ref:crouzeix1973,ref:cohen} of degree 2.
The Lebesgue constant $\Lambda=(187+38\sqrt{19})/243$.
For the stiffness matrix, there are multiple quadrature rules with 6 nodes.
The first one, denoted by \stquad{1}, is the same as ``$d = 3$ (degree), $n_p = 6$ (number of nodes),
type $[0,0,1]$, quality PI'' in \citep{ref:pap}, or $\KK=[0{;}0{,}0{;}0{,}0{,}1]$ in the current notation.
It has a weight $w_1=\myfrac{1}{12}$ and overall error \allerr{33/(4\sqrt{10})}.
The estimated CFL number with numerical quadrature is denoted by $\CFLnum$ and is slightly larger than the $\CFL$ obtained with exact evaluation of the stiffness matrices.

\begin{sloppypar}
The second one, labelled as \stquad{2}, can be found in \citep{ref:pap} as
``$d = 4$, $n_p = 6$, type $[0, 2, 0]$, quality PI'' or $\KK=[0{;}0{,}0{;}0{,}2{,}0]$.
The weights are $w_{1,2}=\{620-[{155 (1375\mp 344 \sqrt{10})}]^{1/2}\}/{7440}$
with node parameters $\{40-5\sqrt{10}\mp [10(95-22 \sqrt{10})]^{1/2}\}/90$.
It produces the same result as the exact stiffness matrices because the degree 4
is one higher than required for accuracy. 
This is the best choice among the infinitely many solutions for this node pattern.
\end{sloppypar}

The other two rules, \stquad{3} and \stquad{4} have node patterns
$[0{;}1{,}0{;}0{,}1{,}0]$ and
$[1{;}0{,}0{;}0{,}1{,}0]$, respectively.
Rule \stquad{3} has weights $w_1=\myfrac{1}{60}$ and $w_2=\myfrac{3}{20}$ with a node parameter $\myfrac{1}{6}$ and error \allerr{33/(4\sqrt{10})}.
The weights for \stquad{4} are $w_1=(1+\sqrt{21})/{240}$ and $w_2=(39-\sqrt{21})/{240}$ with node parameter $(9+\sqrt{21})/{30}$ and
error \allerr{1089(11-\sqrt{21})/500}.

\begin{table}
  \caption{Quadrature rules in 2D for degree 2 ($p_0=p_1=2$, $p_2=3$) with 
    7 nodes for the mass matrix
    and  four different patterns labelled \stquad{1} to \stquad{4} with 6 nodes for the stiffness matrix.
    For the second pattern \stquad{2}, there are infinitely many solutions
    of which the listed one is the best, since it is exact.}\label{tab:twoa}\medskip
{\scriptsize{
     \setlength{\tabcolsep}{6pt} 
  \begin{tabular}{l r l  l  l }
    rule & class & weights & node parameters & remarks \\
    \toprule
\mquad{2,3}& 1& $\myfrac{1}{40}$ & & $\CFL=0.3674$ \\ 
& 2& $\myfrac{1}{15}$ & &$\Lambda=1.451$ \\ 
& 4& $\myfrac{9}{40}$ & &\\ 
\midrule
\stquad{1}& 6& 0.0833333333333333333333333 & 0.109039009072877212324835 &\allerr{2.609}\\ 
& & & 0.231933368553030572496785 &$\CFLnum=0.3844$\\
\stquadopt{2}& 5& 0.0549758718276609338191632 & 0.0915762135097707434595715 & \allerr{0}\\ 
& 5& 0.111690794839005732847504 & 0.445948490915964886318329 &$\CFLnum=0.3674$\\ 
[3pt]\stquad{3}& 2& 0.0166666666666666666666667 & & \allerr{2.609}\\ 
& 5& 0.15& 0.166666666666666666666667 & $\CFLnum=0.3844$\\ 
[3pt]\stquad{4}& 1& 0.0232607320623160000274502 & &\allerr{3.739}\\ 
& 5& 0.143405934604350666639216 & 0.452752523165194666886268 & $\CFLnum=0.3187$\\ 

\marklast\end{tabular} }}\end{table}

Table~\ref{tab:twob} lists the known 12-node triangular element of degree 3 \citep{ref:cohen,ref:tordjmana}.
The weights are
$w_1=\myfraca{8-\sqrt{7}}{720}$,
$w_2=\myfraca{7+4 \sqrt{7}}{720}$ and
$w_3=\myfrac{7 (14-\sqrt{7})}{720}$ with node parameters
$(21-[21(-7+4 \sqrt{7})]^{1/2})/42 $ and $\myfraca{7-\sqrt{7}}{21} $.
The quadrature rule for the stiffness matrix is the same as in \citep{ref:dunavant1985},
apart from a factor $1/\ndim$ for the weights in the current paper. It is the unique solution,
modulo symmetry, of a polynomial system of 5 equations in 5 unknowns,
with weights $w_1=\myfrac{9}{80}$, $w_{2,3}=\myfraca{155\mp\sqrt{15}}{2400}$,
node parameters $(6\mp\sqrt{15})/21$ and error \allerr{\left[ 1313 (198671-30520\sqrt{7})/5\right]^{1/2}\!/1296}.
Exact quadrature of the stiffness matrix can be obtained with 12 nodes and either of the two rules
``$d = 6$, $n_p = 12$, type $[0,2,1]$, quality PI'' in \citep{ref:pap}.

\begin{table}
  \caption{Quadrature rules in 2D for degree 3 ($p_0=p_1=3$, $p_2=4$) with 
    12 nodes for the mass matrix
    and 7 nodes for the stiffness matrix. For the latter, the quadrature rule is the same as in \citep{ref:dunavant1985}.
  }\label{tab:twob}\medskip
  {\scriptsize{ 
    \begin{tabular}{l r l  l  l }
    rule & class & weight & node parameters & remarks \\
    \toprule
\mquad{3,4}& 1& 0.00743645651241029084652553 & &$\CFL=0.2102$ \\ 
& 3& 0.0244208406170255032805645 & 0.293469555909040190389804 &$\Lambda=2.218$\\ 
& 5& 0.110388528920205369259012 & 0.207345175663590924261828 &\\ 
\midrule
\stquad{1}& 4& 0.1125 & &\allerr{4.294}\\ 
& 5& 0.0629695902724135762978420 & 0.101286507323456338800987 &$\CFLnum=0.2139$\\ 
& 5& 0.0661970763942530903688247 & 0.470142064105115089770441 &\\ 

\marklast\end{tabular} }}\end{table}

\begin{table}
  \caption{Quadrature rules in 2D for degree 4 with 
    18 nodes for the mass matrix.
    There are 4 patterns with  15 nodes for the stiffness matrix. The one with 16 nodes, \stquad{5}, is exact.
  }\label{tab:twoc}\medskip
  {\scriptsize{ 
    \begin{tabular}{l r l  l  l }
    rule & class & weight & node parameters & remarks \\
    \toprule
\mquad{4,5}& 1& 0.00317460317460317460317460 & &$\CFL=0.1281$ \\ 
& 2& 0.0126984126984126984126984 & &$ \Lambda=2.650$ \\ 
& 3& 0.0107142857142857142857143 & 0.211324865405187117745426 &\\ 
& 5& 0.0505838648956875558102522 & 0.130791593829744967194355 &\\ 
& 5& 0.0787812144693918092691129 & 0.424763961725810588361201 &\\ 
\midrule
\stquad{1,1}& 5& 0.0265389008951162058359775 & 0.0649305131591648630783798 &\allerr{4.941} \\ 
& 6& 0.0346373410397084467561383 & 0.0438634717923724715117987 & $\CFLnum=0.1344$\\ 
& & & 0.313559184384931507955852 &\\
& 6& 0.0354265418460667836592063 & 0.198384476681506719179877 &\\ 
& & & 0.284575584249170335197416 &\\
\stquad{1,2}& 5& 0.0626968037246515334107852 & 0.243259139835607536562462 &\allerr{5.429}\\ 
& 6& 0.0138317623007367135917420 & 0.0457208298463203239115760 &$\CFLnum=0.1371$\\ 
& & & 0.0866366313417490017137638 &\\
& 6& 0.0381531691702708530361988 & 0.0507143843072070430506436 &\\ 
& & & 0.318644189847537054206171 &\\
[3pt]\stquadopt{2}& 5& 0.00598871750086874787313223 & 0.0248873175547711681618438 &\allerr{1.786}\\ 
& 5& 0.0645271001594989722247439 & 0.240852354623800158854260 &$\CFLnum=0.1304$\\ 
& 5& 0.0399437127829381765035879 & 0.474313723462548969585082 &\\ 
& 6& 0.0281035681116803850326013 & 0.0457724799258222965927693 &\\ 
& & & 0.764938750987214477570869 &\\
\stquad{3}& 3& 0.0151405082189550635241782 & 0.311177333051762799603027 &\allerr{12.85}\\ 
& 5& 0.0102531798086683692362102 & 0.0359210001237433518114575 &$\CFLnum=0.09904$\\ 
& 5& 0.0481564641342274957821757 & 0.136412325984609344710820 &\\ 
& 5& 0.0779760062858606745999243 & 0.424933799107976639983030 &\\ 
[3pt]\stquad{4}& 1& 0.00256395435230237177356845 & & \allerr{2.584}\\ 
& 5& 0.0663059700986529548882131 & 0.238561530018070755823421 & $\CFLnum=0.1219$\\ 
& 5& 0.0419943898833667692444273 & 0.474388086175154277132999 &\\ 
& 6& 0.0279011761661722853802289 & 0.0421382841642368479501656 &\\ 
& & & 0.173896050734549278928165 &\\
[3pt]\stquad{5}& 4& 0.0721578038388935841255456 & &\allerr{0}\\ 
& 5& 0.0162292488115990401554630 & 0.0505472283170309754584236 &$\CFLnum=0.1281$\\ 
& 5& 0.0516086852673591251408958 & 0.170569307751760206622294 &16 nodes\\ 
& 5& 0.0475458171336423123969481 & 0.459292588292723156028816 & \\ 
& 6& 0.0136151570872174971324223 & 0.00839477740995760533721383 &\\ 
& & & 0.263112829634638113421786 &\\

  \marklast\end{tabular} }}
\end{table}

For the triangular degree-4 element found in \cite{ref:mulder1996}, Table~\ref{tab:twoc}
repeats the quadrature rule for the mass matrix. The weights are
$w_1=\myfrac{1}{315}$, $w_2=\myfrac{4}{315}$, $w_3=\myfrac{3}{280}$, $w_{4,5}=\myfraca{1141\mp 94 \sqrt{7}}{17640}$
with node parameters $\left(5\mp\sqrt{7}\right)/18$.
The table lists 5 rules for the stiffness matrix, four with 15 nodes and one with 16.
For the node pattern $\KK=[0{;}0{,}0{;}0{,}1{,}2]$, there are two
solutions, modulo symmetry, with \stquad{1,1} the more accurate one.

There are 8 equations with 9 parameters for node pattern $\KK=[0{;}0{,}0{;}0{,}3{,}1]$. Optimization produced
\stquadopt{2}, which hopefully is the global minimum.

The pattern $\KK=[0{;}0{,}1{;}0{,}3{,}0]$ leads to a unique solution \stquad{3},
as does $\KK=[1{;}0{,}0{;}0{,}2{,}1]$ for \stquad{4}.
For 16 instead of 15 nodes, the rule \stquad{5} from \citep{ref:pap} is exact. With the accuracy criterion used here, it represents one of many solutions. It is unique for the requirement that quadrature be exact for polynomials up to degree 8.

\subsection{Tetrahedron}

\begin{sloppypar}
For 3D, a new degree-2 element with 14 nodes is listed Table~\ref{tab:threea}.
There are infinitely many solutions of the form
\be{eq:tet205}
w_1=\tfrac{1}{24}-\tfrac{3}{2}w_2 - w_3,\quad w_2=\tfrac{1}{30}-8 \al (1-2 \al),\quad
w_3=1/[20160\,\al^3 (1-3 \al)],
\ee
as a function of the  single node parameter $a$. 
Weights are positive 
in the range
${\scriptstyle{0.301980838336061169024588}}<a<{\scriptstyle{0.319260087816400798671151}}$.
At the lower bound of the interval,
the weight $w_1=0$, whereas $w_2=0$ at the upper bound.
The entry in Table~\ref{tab:threea} corresponds to the largest CFL number.
The estimated CFL number is small, compared to the known 15-node element in Table~\ref{tab:three}, making this element less attractive for practical applications.
The difference is smaller when using the 10-node quadrature rule for the stiffness matrix, which leads to a significantly larger estimate of $\CFLnum$.
The weights for the latter are $w_1=1/360$ and $w_2=7/270$ with a node parameter $\myfraca{7-\sqrt{21}}{28}$.

\end{sloppypar}

\rmrk{Try in cml1?}

Table~\ref{tab:three} contains results for the known element with degrees $\{2,3,4\}$  \citep{ref:geevers_new}, which
is simpler than the one with the element degrees $\{2,4,4\}$ found much earlier \cite{ref:mulder1996}.
For the stiffness matrices, the rule \stquad{1} for the element with degrees 2,-,5 in Table~\ref{tab:threea} is also a solution for the current element, but is not spurious-free and has to be discarded.
There are several rules with 14 nodes in Table~\ref{tab:three}.
The first one, \stquad{1,1}, can be found in \citep{ref:grundmann} and was used in \citep{ref:geevers_stiff}.
The polynomial system consists of 4 equations in 6 unknowns.
Its solution can be expressed as function of the first two node parameters,
split out into many different branches if positivity is required.
Minimisation produced \stquadopt{1,2}, with the caveat that it might be a local minimum.

\rmrk{describe other, clean up labelling}

\begin{sloppypar}
The second rule, \stquad{2},
requires the solution of 4 equations in 5 parameters.
Positive solutions are found if the first node parameter obeys $0<a\le {\scriptstyle{0.137285957672957219851842 }}$.
Table~\ref{tab:three}  has an optimized result that is close to the upper bound.

The third rule, \stquad{3}, involves 4 equations in 5 unknowns.

The solutions can be expressed as a function of the first node parameter $a$.
Optimization over the other two solution branches, with
$0<a<a_1$ or 
$a_1<a< {\scriptstyle{0.137204354363517160663447}}$
produced the rule \stquadopt{3,1}.
For $a=a_1=(3-\sqrt{2})/14$, the rule \stquad{3,2} in Table~\ref{tab:three} is obtained,
with weights $w_1=\myfrac{1}{360}$ and $w_2=\myfraca{108\pm\sqrt{2}}{5760}$
and node parameters $\myfraca{3\mp\sqrt{2}}{14}$.
\end{sloppypar}

\rmrk{try in cml1?}

The quadrature rule \stquad{4} has a unique solution, modulo symmetry, and so has \stquad{26}.

\begin{sloppypar}
The quadrature rule \stquad{5} for the stiffness matrix involves 4 equations in 5 unknowns
and has positive solutions for
$\myfraca{1+2 \sqrt{2}}{630}<w_3<\myfrac{7}{270}$ or
${\scriptstyle{0.00607686845197807952000536}} < w_3 < {\scriptstyle{0.0259259259259259259259259}}$.
The result in Table~\ref{tab:three} was obtained by 1-D minimisation.
At the upper bound of $w_3$, $w_2=0$. In addition, the last node of the class $[2,2]$ with index 9 becomes $(1/2,1/2,0,0)$
and moves to a different class, the edge midpoint.
This case is labelled \stquad{6} in the table, with a unique solution and less favourable  $\errmatall$.
Its weights are  $w_1=\myfraca{8\sqrt{2}-3}{2520}$, $w_2=\myfraca{1+2 \sqrt{2}}{630}$ and $w_3=\myfraca{51-10\sqrt{2} }{1260}$
with node parameter  $\myfraca{3+\sqrt{2}}{14}$.
\end{sloppypar}

The tetrahedral element of degree 3 in Table~\ref{tab:threethree} is also known \citep{ref:geevers_new}, as is the first quadrature rule, \stquad{1},
for the stiffness matrices \citep{ref:geevers_stiff}.
There are actually two solutions, modulo symmetry and the second one, \stquad{2}, has a slightly smaller value of $\errmatall$.
The weights for the mass matrix are
$w_1=\myfraca{41-9 \sqrt{2}}{41160}$,
$w_2=\myfraca{8+9 \sqrt{2}}{13720}$,
$w_3=\myfraca{10-\sqrt{2}}{1715}$
and $w_4=\myfrac{3}{140}$.
The node parameters are, subsequently,
$\myfraca{ 3-[3(\sqrt{2}-1)]^{1/2} }{6}$,
$\myfraca{ 4 - \sqrt{2} }{12}$ and $\myfrac{1}{6}$.

\begin{table}
  \caption{A degree-2 tetrahedral element ($p_0=p_1=2$, $p_3=5$) with 14 nodes 
    in 3D.
    For the stiffness matrix, 10 nodes suffice. 
  }\label{tab:threea}\medskip
  {\scriptsize{
   \begin{tabular}{l  r l  l  l }
   rule & class & weight & node parameters & remarks \\
   \toprule
\mquadopt{2,\text{-},5}& 1& 0.00106024229121144047613644 & &$\CFL=0.1471$\\ 
& 2& 0.0121215111478839094743360 & & $\Lambda=3.649$\\ 
& 8& 0.0224241576536293619790263 & 0.308129066037442904319390 &\\ 
\midrule
\stquad{1}& 1& 0.00277777777777777777777778 & &\allerr{14.92}\\ 
& 9& 0.0259259259259259259259259 & 0.0863365823230057140504269 &$\CFLnum=0.2321$\\
\marklast
  \end{tabular} }}
\end{table}
\begin{table}
  \caption{A known degree-2 tetrahedral element with 15 nodes 
    has a more favourable CFL number than the 14-node degree-2 element. For the stiffness matrix,
    several 14-node integration rules are listed with the smallest error by far for \stquad{1,2}.
   }\label{tab:three}\medskip
   {\scriptsize{
    \begin{tabular}{l r  l  l  l }
    rule & class & weight & node parameters & remarks \\
\toprule
\mquad{2,3,4}& 1& 0.00337301587301587301587302 & &$\CFL=0.3284$\\ 
& 2& 0.00634920634920634920634921 & &$\Lambda=1.625$\\ 
& 4& 0.0160714285714285714285714 & &\\ 
& 7& 0.0507936507936507936507937 & &\\ 
\midrule
\stquad{1,1}& 8& 0.0122488405193936601769620 & 0.0927352503108912348701559 &\allerr{0.3799}\\ 
& 8& 0.0187813209530026406750659 & 0.310885919263300609915663 &$\CFLnum=0.3260$\\ 
& 9& 0.00709100346284691054309247 & 0.0455037041256496494057301 & \\ 
\stquadopt{1,2}& 8& 0.0110792758509461122468093 & 0.0880772159733335667633104 &\allerr{0.0448}\\ 
& 8& 0.0168267504736189121759326 & 0.309995447400425430662913 &$\CFLnum=0.3284$\\ 
& 9& 0.00917376022806776149594985 & 0.0576097472208336369395160 &\\ 
[3pt]\stquadopt{2}& 4& 0.00579346744884706261068342 & &\allerr{0.8878}\\ 
& 8& 0.0326121866845538857644211 & 0.137241091530281438591289 &$\CFLnum=0.3414$\\ 
& 9& 0.00217400835551047886104146 & 0.00368270955954534123049295 &\\ 

\stquadopt{3,1}& 2& 0.00221832976797668744396797 & &\allerr{0.8433}\\ 
& 8& 0.0283607853121208351698735 & 0.131447443461074085567475 &$\CFLnum=0.3366$\\ 
& 8& 0.00997838670258080033084118 & 0.322816586241320942220683 &\\ 
\stquad{3,2}& 2& 0.00277777777777777777777778 & &\allerr{1.316}\\ 
& 8& 0.0189955231879119956681947 & 0.113270459830493210799879 &$\CFLnum=0.3170$\\ 
& 8& 0.0185044768120880043318053 & 0.315300968740935360628692 &\\ 
[3pt]\stquad{4}& 2& 0.00205810944851546220477789 & &\allerr{0.8926} \\ 
& 4& 0.00594005740862085571487759 & &$\CFLnum=0.3405$ \\ 
& 8& 0.0326394450852726176446222 & 0.137204354363517160663447 &\\ 
\stquadopt{5}& 1& 0.00298957604762401491577617 & &\allerr{1.179}\\ 
& 8& 0.0162611487338547701734117 & 0.308820729178200911373427 & $\CFLnum=0.3053$\\ 
& 9& 0.0149439612567919210516525 & 0.0572083410744224425608908 &\\ 
\stquad{6}& 1& 0.00329909067420030174222758 & &\allerr{4.086}\\ 
& 2& 0.00607686845197807952000536 & &$\CFLnum=0.2588$\\ 
& 8& 0.0292522733144992456444310 & 0.315300968740935360628692 &\\ 

  \marklast\end{tabular} }}
\end{table}

\begin{table}
  \caption{A known tetrahedral element of degree 3 with 
    32 nodes for the mass matrix and 
    21 nodes for the stiffness matrix. 
  }\label{tab:threethree}\medskip
  {\scriptsize{ \begin{tabular}{l  r  l  l  l }
  rule & class & weight & node parameters & remarks \\
  \toprule
\mquad{3,4,5}& 1& 0.000686882360025319352788746 & &$\CFL=0.1906$\\ 
& 3& 0.00151078149135261337020519 & 0.314210342418032893883175 &$\Lambda=2.928$\\ 
& 5& 0.00500628946800402621061126 & 0.215482203135575412599859 & \\ 
& 8& 0.0214285714285714285714286 & 0.166666666666666666666667 &\\ 
\midrule
\stquad{1}& 7& 0.0189417739968774012264957 & &\allerr{2.531}\\\ 
& 8& 0.0106280309733063567967261 & 0.319555604693565569996655 &$\CFLnum=0.1930$\\ 
& 8& 0.00838281346260630865013490 & 0.0836098229399537906599478 & \\ 
& 10& 0.00597345957717821697106059 & 0.0636610018750175252992355 &\\ 
& & & 0.336251922239849416310752 &\\
[3pt]\stquad{2}& 7& 0.0175734672984752953111294 & &\allerr{2.128}\\ 
& 8& 0.00123565503087609417775721 & 0.0313142510249888223084313 &$\CFLnum=0.1915$\\ 
& 8& 0.0111960044261643789340031 & 0.318387168152114882888947 &\\ 
& 10& 0.00828054679500245657570801 & 0.0636610018750175252992355 &\\ 
& & & 0.264779316613278444841917 &\\

  \marklast\end{tabular} }}
\end{table}
\clearpage\newpage

\subsection{4-simplex}

\rmrk{Are all cases for degree 3 scanned?}

In 4 space dimensions, there is a degree-2 element with 20 nodes, listed in Table~\ref{tab:threetwoa}.
The polynomials are elements of $P_2\oplus \Permfun{\bub_3}$ and
the nodes are the five vertices, midpoints of the ten edges, and centroids of the five 3-faces or facets.
The corresponding degrees are $p_0=p_1=2$ and $p_4=4$.
The weights are $w_1=\myfrac{1}{2520}$, $w_2=\myfrac{1}{1260}$ and $w_3=\myfrac{2}{315}$
and the Lebesgue constant is $\Lambda=493/125$.

For this element, the numerical quadrature of the stiffness matrix has the same number of nodes as the rule for the mass matrix. Rules with less nodes were not found.

The optimised stiffness rule \stquadopt{1} involves 4 equations in 6 parameters. 
\begin{sloppypar}
For the system of 4 equations in 8 unknowns with node pattern $\KK=[0{;}0{,}0{;}0{,}0{,}0{;}0{,}0{,}0{,}0{,}0{;}0{,}4{,}0{,}0{,}0{,}0{,}0]$,
a solution could not be found but may exist.
\end{sloppypar}

Rule \stquadopt{2} 
involves  4 equations in 5 unknowns. The solution can be expressed as functions
of the first node parameter with $0 < a \le {\tiny }0.138938103986157644354190$ and $a\ne \myfraca{6-\sqrt{6}}{30}$. The listed result has the smallest $\errmatall$.

The stiffness rule with 
$\KK=[0{;}1{,}0{;}0{,}0{,}0{;}0{,}0{,}0{,}0{,}0{;}0{,}2{,}0{,}0{,}0{,}0{,}0]$ has many solutions. Optimisation
produces a degenerate solution which happens to be the rule \stquad{3} with weights
$w_1=\myfrac{1}{2520}$, $w_2=\myfrac{4}{945}$ and $w_3=\myfrac{5}{1512}$ for a node parameter $1/10$.

Rule \stquadopt{4} is obtained for 
4 equations into 7 unknowns.
The solution for rule \stquad{5} 
is unique,
as the one for \stquad{6}, which is the same as the quadrature rule for the  mass matrix.

Table~\ref{tab:threetwo} contains the weights for a second degree-2 element.
It has degree 2 on the edges, 3 on the faces, 4 on the facets or 2-faces, and 5 in the interior.
The weights are $w_1=\myfrac{149}{362880}$, $w_2=\myfrac{13}{22680}$, $w_3=\myfrac{1}{896}$, $w_4=\myfrac{8}{2835}$ and $w_5=\myfrac{625}{72576}$.
The corresponding nodes are the centroids of the $m$-faces, for $m=0,1,\ldots,4$.

Several quadrature rules for the stiffness matrix with 30 nodes were found, but the existence of rules with less nodes cannot be excluded and
many of the polynomial systems for the 30-node rules were too difficult to solve in a reasonable amount of time.

\newcommand{\ncix}{1}   
\newcommand{\ncxii}{2}  
\newcommand{\ncxxiv}{3} 
\newcommand{\ncxxx}{4}  
Rule \stquadopt{\ncix} involves 5 equations in 6 unknowns and only the optimised result is listed in Table~\ref{tab:threetwo}.
Rules \stquad{\ncxii} and \stquad{\ncxxiv} involve 5 equations in 5 unknowns with a unique solution,
as does rule \stquad{\ncxxx} with 
weights $w_1=\myfraca{33 \sqrt{5}-41}{80640}$,
$w_2=\myfraca{3 \sqrt{5}-1}{10080}$, $w_3=\myfrac{3 \left(1+\sqrt{5}\right)}{8960}$ and
$w_4=\myfrac{3 \left(5-\sqrt{5}\right)}{1792}$ with node parameter $\myfraca{5+\sqrt{5}}{30}$.

\newcommand{\nxxxv}{2} 
\newcommand{\nxxxvii}{1} 
Finally, Table~\ref{tab:three3} lists weights and node parameters for a degree-3 element in 4 space dimensions, with 80 nodes
for the mass matrix and 46 for the stiffness matrix.
For the quadrature of the mass matrix,
some parameters have a simple representation, for instance,
\[  a_{8,1,1}=\left\{1-[ (7+\myfrac{8}{\sqrt{11}})/95]^{1/2}\right\}/4,\quad \al_{13,1,1}=\myfraca{1-1/\sqrt{11}}{5},\]
\[  w_5=\myfrac{121(11+56\sqrt{11})}{7257600}.
\]
Other entries are too long to be reproduced over here.

The node pattern $\KK=[0;0,0;0,0,0;0,0,0,0,0;1,3,1,1,0,0,0]$ with 46 nodes
leads to infinitely many quadrature rules for the stiffness matrix.
Optimization with respect to $\errmatall$ lets one of three nodes in class $[4,1]$ ($K_{13}$) move to
class $[1]$ ($K_1$) and produces the same result as \stquad{\nxxxvii,1}.

\rmrk{more stiffness cases?}

\rmrk{Are there more??}

\rmrk{$\{3,4,4,7\}$??}
\begin{table} 
  \caption{An element of degree 2 ($p_0=p_1=2$, $p_4=4$) in 4D 
    with 20 nodes
    for the mass matrix and 20 for the stiffness matrix.}\label{tab:threetwoa}\medskip
  {\scriptsize{
    \begin{tabular}{l r  l  l  l }
        rule & class & weight & node parameters & remarks \\
        \toprule
\mquad{2,\text{-},4,\text{-}}& 1& 0.000396825396825396825396825 & &$\CFL=0.1789$\\\ 
& 2& 0.000793650793650793650793651 & &$\Lambda=3.944$\\ 
& 7& 0.00634920634920634920634921 & &\\ 
\midrule
\stquadopt{1}& 13& 0.000463047128668606081555787 & 0.0235803592776537401257936 &\allerr{0.2982}\\ 
& 13& 0.00193826296349263899875603 & 0.244634182009346471232919 &$\CFLnum=0.1804$\\ 
& 14& 0.00296601162058604412651076 & 0.0705403895332606336540964 &\\ 
\stquadopt{2}& 7& 0.000889938669226683906602000 & &\allerr{0.3284} \\ 
& 13& 0.000747942352712646481290838 & 0.0493624021191759830090195 &$\CFLnum=0.1902$\\ 
& 14& 0.00334772615569700147272025 & 0.0796632726401685365816004 &\\ 
[3pt]\stquad{3}& 2& 0.000396825396825396825396825 & &\allerr{0.9810}\\ 
& 7& 0.00423280423280423280423280 & &$\CFLnum=0.1362$\\ 
& 13& 0.00330687830687830687830688 & 0.1 &\\ 
[3pt]\stquadopt{4}& 1& 0.000311324842891441479609980 & &\allerr{0.3013}\\ 
& 13& 0.00227335460158547594966056 & 0.243317819101603199938769 &$\CFLnum=0.1804$\\ 
& 14& 0.00287432694442820795203140 & 0.0657782802911012693956299 &\\ 
[3pt]\stquad{5}& 1& 0.000291032680706059271471708 & &\allerr{0.4496} \\ 
& 7& 0.000253880391325866013076778 & &$\CFLnum=0.2097$\\ 
& 14& 0.00389421013065070402439242 & 0.0782690441575398487439937 &\\ 
[3pt]\stquad{6}& & same as \mquad{2,\text{-},4,\text{-}} & &\allerr{1.965}\\ 
& & & & $\CFLnum=0.1117$ \\ 

\marklast\end{tabular}}}
\end{table}
\begin{table} 
  \caption{A second element of degree 2 has 
    31 nodes in four space dimensions for the mass matrix and 30 for the stiffness matrix.}\label{tab:threetwo}\medskip
   {\scriptsize{
    \begin{tabular}{l  r  l  l  l }
    rule & class & weight & node parameters & remarks \\
\toprule

\mquad{2,3,4,5}& 1& 0.000410604056437389770723104 & &$\CFL=0.3017$\\ 
& 2& 0.000573192239858906525573192 & & $\Lambda=1.803$\\ 
& 4& 0.00111607142857142857142857 & &\\ 
& 7& 0.00282186948853615520282187 & &\\ 
& 12& 0.00861166225749559082892416 & &\\ 
\midrule
\stquadopt{\ncix}& 2& 0.000191852529355864471731846 & &\allerr{0.3084}\\ 
& 7& 0.000789815419126082369412440 & &$\CFLnum=0.3093$\\ 
& 13& 0.00546419545327554614783932 & 0.115795357027714070980707 &\\ 
& 14& 0.000847808701109987936308939 & 0.319511843067379711333355 &\\ 
\stquad{\ncxii}& 2& 0.000197603167391832207462443 & &\allerr{0.5908}\\ 
& 4& 0.000539373841340480257392759 & &$\CFLnum=0.2965$\\ 
& 7& 0.00123494437899886025457195 & &\\ 
& 13& 0.00562443493686984814905098 & 0.116358009970857270976190 &\\ 
\stquad{\ncxxiv}& 1& 0.000291064526770509063833014 & &\allerr{0.7033} \\ 
& 4&   $9.59219805561710955247146\!\cdot\!10^{-6}$ & &$\CFLnum=0.3085$\\ 
& 7& 0.000216344774919483390481248 & &\\ 
& 14& 0.00390336981776605332995706 & 0.0784615942469140643522620 &\\ 
\stquad{\ncxxx}& 1& 0.000406625040395499255723000 & &\allerr{2.061} \\ 
& 2& 0.000566290072668588203296381 & & $\CFLnum=0.2195$\\ 
& 4& 0.00108350490318073315727986 & &\\ 
& 13& 0.00462711834123919135645786 & 0.241202265916659656546972 &\\ 

\marklast\end{tabular}}}
\end{table}

\begin{table} 
  \caption{A degree-3 element in 4 space dimensions with 
   80 nodes.
   A 46-nodes quadrature rule for the stiffness matrix, \stquad{\nxxxvii}, has two solutions.
For \stquad{\nxxxv}, also with 46 nodes, one of infinitely many solutions is shown.
Minimisation of $\errmatall$ produces \stquad{\nxxxvii,1}.
} 
  \label{tab:three3}\medskip
   {\scriptsize{
    \begin{tabular}{l r  l  l  l }
    rule & class & weight & node parameters & remarks \\
    \toprule
\mquad{3,4,5,6}& 1& 0.0000634687423149349589944868 & &$\CFL=0.1767$\\ 
& 3& 0.0000973913124562213436677687 & 0.339361359211720168522736 & $\Lambda=3.907$\\ 
& 5& 0.000236701596558937161759499 & 0.224923912791330495226026 &\\ 
& 8& 0.000795038920444809458280685 & 0.171309689469020940942630 &\\ 
& 13& 0.00327993408006065219598804 & 0.139697731084447275470638 &\\ 
\midrule
\stquad{\nxxxvii,1}& 1& 0.0000135796959834559115220899 & &\allerr{0.9307}\\ 
& 12& 0.00251361779741206305919335 & &$\CFLnum=0.1789$\\ 
& 13& 0.000143411908756882943527805 & 0.0469282497504013574209646 &\\ 
& 13& 0.00130249665210575976369628 & 0.242337281353804306909773 &\\ 
& 14& 0.00114995179198013822391338 & 0.303785526045443074349591 &\\ 
& 15& 0.00101780448326113641373043 & 0.0598508375909205620140451 &\\ 
& & & 0.253184114259116209547217 &\\
\stquad{\nxxxvii,2}& 1& 0.0000143649822356490473801218 & &\allerr{0.9434}\\ 
& 12& 0.00252384195775506540526455 & & $\CFLnum=0.1781$\\ 
& 13& 0.000435289306495329172060708 & 0.0678530017485141894267718 &\\ 
& 13& 0.00129937904955452706670238 & 0.242370040012056372359001 &\\ 
& 14& 0.00114995179198013822391338 & 0.303785526045443074349591 &\\ 
& 15& 0.000944907004884134629577612 & 0.0598508375909205620140451 &\\ 
& & & 0.265490699387678247005096 &\\
\stquad{\nxxxv,1}& 12& 0.00249856167458254646466376 & &\allerr{0.9323}\\ 
& 13& 0.0000237672017911754952149392 & 0.00722636396987802183985101 &$\CFLnum=0.1780$\\ 
& 13& 0.000196546166835736134567618 & 0.0579425769792456787565224 &\\ 
& 13& 0.00130664508237860938292903 & 0.242292938358339861353259 &\\ 
& 14& 0.00114995179198013822391338 & 0.303785526045443074349591 &\\ 
& 15& 0.00100168974086275664496556 & 0.0598508375909205620140451 &\\ 
& & & 0.255706751036907967802450 &\\
\stquadopt{\nxxxv}=\stquad{\nxxxvii,1}\\

\marklast\end{tabular}}}
\end{table}

\section{Conclusions}\label{sec:conc}

Mass lumping of mass matrix in the finite-element discretization of the wave equation allows for explicit time stepping.
The entries in the diagonal lumped mass matrix are equal or proportional to numerical quadrature weights.
Higher-degree polynomials on the $m$-faces, with $m>1$, are needed to preserve the spatial accuracy after lumping on the $n$-simplex, with $n>1$.

New quadrature rules for the 4-simplex have been presented.
A new tetrahedral degree-2 element has one node less than a known element,
but also has a considerably smaller CFL number, which determines the maximum allowable time step.
This makes the element less favourable in spite of its slightly lower computational cost.

Given a mass-lumped element, numerical quadrature of the stiffness matrices generally is more efficient than their exact evaluation.
New dedicated quadrature rules for elements on the 2- and 4-simplex were presented, and some additional rules beyond the known ones on the 3-simplex.

The construction of a quadrature rule involves the solution of a polynomial system of equations. Only those with positive weights are of interest.
If there are multiple solutions, the one with the largest CFL number is selected.
Likewise, if there are multiple solutions for the quadrature equations related to the stiffness matrices, the one with smallest error
in the sum of the squared Frobenius norms of the individual matrices can be chosen.

The search for quadrature rules, in particular for the stiffness matrices, involved many cases. The existence or non-existence of positive
solutions of the polynomial systems could not always be established, in particular for the higher degrees and dimensions. Therefore, there may exist other and less costly rules.

Simpler elements with less nodes can be found by considering subsets of the higher-degree polynomials used to enrich the elements to preserve accuracy, as in \citep{ref:geevers_new} for tetrahedral elements of degree 4. Here, this approach has not been considered.

Weak instead of strict face-conformity may also lead to elements with less nodes,
although the proposed quadratic element on the 4-simplex in \cite{ref:petrov2023} actually requires 21 nodes, whereas one of the two quadratic elements over here has only 20.

\appendix

\section{Simplification of the vertex monomial}\label{app:lower}

\begin{theorem}{
    In the definition of $U$ in equation~\eqref{eq:polspace},
    $V_0=\left\{ x_0^{p_0} \right\}$ can be replaced by $V_0=\left\{ x_0 \right\}$ if $p_0=p_1\ge 2$
    and $P_{p_1}\subset U$.
}\end{theorem}
\begin{proof}{
Because $\sum_{k=0}^\ndim x_k=1$, $x_0^{q-1}=x_0^{q-1}\sum_{k=0}^\ndim x_k= x_0^{q}+x_0^{q-1}\sum_{k=1}^\ndim x_k$
or $x_0^{q} = x_0^{q-1}-x_0^{q-2}\sum_{k=1}^\ndim x_0 x_k$ for $q\ge 2$.
Therefore, $x_0^q \in \pspan{ \left\{  x_0^{q-1} \right\} \cup \Permfun{  V_{1,q} } }$.
Here, $V_{1,q}=\left\{\bub_1\right\}\otimes P_{q-2}(x_0,x_1)$,
that is, $V_1$ for $p_1=q$ at $m=1$ in equation~\eqref{eq:vv}.
Starting at $q=p\ge 2$ and using  $V_{1,q-1}\subset \pspan{V_{1,q}}$,
recursion with decreasing $q$ ends at
$x_0^ 2\in \pspan{ \left\{x_0\right\} \cup \Permfun{  \left\{\eta_1\right\} }  }$.\
The right-hand side is a subset of $\left\{ x_0 \right\} \cup \Permfun{ V_{1,q} }$.
This implies that $x_0^{p_1}\in \pspan{ \left\{ x_0 \right\} \cup \Permfun{ V_{1,p_1} } }$.

The reverse, $x_0\in \pspan{ \left\{ x_0^{p_1} \right\} \cup \Permfun{ V_{1,p_1} } }$, follows in a similar way.
}\end{proof}

\rmrk{SOMETHING MISSING? The $x_k$ causes problems for $k>1$? Is $P_{p_1}\subset U$ needed? do example?}
\rmrk{reference to degree lifting or degree raising or degree elevation? Farin?   Dahmen en Micchelli? Szafnicki?}

\bibliographystyle{elsarticle-num}
\bibliography{refs}
\end{document}